\documentclass[11pt,a4paper,twoside]{amsart}
\usepackage{amssymb}
\usepackage{mathtools}
\usepackage{graphics}
\usepackage{bm}
\usepackage{tikz}
\usepackage{xstring}
\usepackage{wasysym}
\usepackage{mathdots} 
\usepackage{rotating}
\usepackage{todonotes}
\usepackage{youngtab}
\usepackage{tikz}
\usetikzlibrary{calc}

\newtheorem{theorem}{Theorem}[section]
\newtheorem{lemma}[theorem]{Lemma}
\newtheorem{proposition}[theorem]{Proposition}

\newtheorem{conjecture}[theorem]{Conjecture}
\theoremstyle{remark}
\newtheorem{remark}[theorem]{Remark}
\newtheorem{example}[theorem]{Example}

\renewcommand\P{\operatorname{\mathbb P{}}}

\newcommand{\refT}[1]{Theorem~\ref{#1}}

\newcommand{\refP}[1]{Proposition~\ref{#1}}

\newcommand{\refS}[1]{Section~\ref{#1}}

\newcommand{\vac} {\ensuremath{\cdot}}
\newcommand{\occ} {\ensuremath{\ \bigcirc}}

\newcommand{\con}{\mbf{q}}

\newcommand{\mbf}{\mathbf}
\newcommand{\mbb}{\mathbb}

\newcommand{\SD}{\Xi}

\newcommand\blfootnote[1]{%
  \begingroup
  \renewcommand\thefootnote{}\footnote{#1}%
  \addtocounter{footnote}{-1}%
  \endgroup
}

\begin{document}
\title{Continuous multi-line queues and TASEP}
\author{Erik Aas}
\address{Department of Mathematics, KTH - Royal Institute of Technology \\
SE-100 44 Stockholm, Sweden}
\email{eaas@kth.se}
\author{Svante Linusson}
\email{linusson@kth.se}
\date{}

\begin{abstract}
In this paper, we study a distribution $\SD$ of labeled particles on a continuous ring. It arises in three different ways, all related to the multi-type TASEP on a ring. We prove formulas for the probability density function for some 
permutations and give conjectures for a larger class. We give a complete conjecture for the probability of two particles 
$i,j$ being next to each other on the cycle, for which we prove some cases. We also find that two natural events 
associated to the process have exactly the same probability expressed as a Vandermonde determinant. 
It is unclear whether this is just a coincidence or a 
consequence of a deeper connection.
\end{abstract}

\maketitle
\section{Introduction}

\blfootnote{The authors acknowledge financial support from the Swedish research council, grant 621-2009-5864.}

In this paper, we study a distribution $\SD$ of labeled particles on a continuous ring. We call this distribution continuous TASEP on a ring. It arises in three different ways:
\begin{enumerate}
\item{} As the limit of the stationary distribution of the totally asymmetric exclusion process (TASEP) on a ring.
\item{} As the projection to the last row of a random continuous multiline queue.
\item{} As the stationary distribution of the so-called (continuous) process of the last row.
\end{enumerate}

Exact definitions are given in \refS{S:Back}. The equivalence of these three descriptions in the discrete case follows from the seminal work by Ferrari and Martin \cite{FM2}. 
The first two explicitly and the third implicitly, as described in Section \ref{S:last_row}.
The limit of the TASEP considered here keeps the number of jumping particles constant and lets the number of vacant positions tend to infinity. 

A large number of TASEP's have been studied, from probabilistic, combinatorial and physical viewpoints. 
In previous works by several authors e.g. \cite{AS, AM,AL,FM2, Lam, LW}, the discrete version of the TASEP studied here has been proven to have many remarkable properties and to be connected to other objects: the shape of random $n$--core partitions, random walks in an affine Weyl group, and multiline queues.
In the present paper we study properties of the limit distribution $\SD$ and find some unexpectedly nice properties of it.
For simplicity, most of our results are stated for the case when the $n$ particles are labeled by $\{1,\dots,n\}$, corresponding to a cyclic permutation.

\subsection{Results and conjectures}
If we condition on the permutation $\pi$ of the distribution we can in some special cases give an exact description of the density function $g_{\pi}$ of how the particles are located on the circle, see \refS{S:Densities}.
For the reverse permutation $w_0=n\dots321$, the density function is the Vandermonde determinant, see Theorem \ref{T:w0}.
For a class of other permutations the density function is, mostly conjecturally, an explicit linear combination of derivatives of the Vandermonde determinant. See \refT{T:Cont12} and Conjectures \ref{Conj:Pdfoneaway} and \ref{Conj:Pdfmanyaway}.
In the cases we can prove, we first prove an exact formula for the discrete case with a given number of empty sites and then take the limit, see Section \ref{S:Disc chain}.
An interesting observation is that the density functions $g_\pi$ in several cases satisfy Laplace's equation.
In general, it seems difficult to give a closed formula for the probability that the particles form a given permutation 
and we have no general conjecture.
However, we give a closed formula for the probability for $w_0$, see Theorem \ref{T:ExactProb}.

Despite the difficulty we had in understanding the probability of a permutation, it seems to be within reach to study certain two-point correlations corresponding to adjacency in the permutation, that is, the probability of two given labeled particles being next to each other. The corresponding
 two point correlation turned out to be important in the discrete case \cite{AAV,AL2}, and has been the second focus of our study. In \refS{S:Correlations} we present a tantalizing pattern for this correlation that we formulate as a general conjecture.
We prove some special cases of the conjecture. 

In \refS{S:initial} we study a third problem, namely the probability that $k$ particles adjacent to each other form a descending sequence in the discrete chain. This is proven, \refT{T:initial}, to be the same Vandermonde determinant 
as in \refT{T:w0}, which settles Conjecture 8.1 in \cite{AL2}. 
We don't know if these two different probabilities expressed as a Vandermonde determinant is just a coincidence, or if there is a deeper connection. See Remark \ref{R:5} for a discussion.

\medskip
\noindent
{\bf Acknowledgement} We thank Andrea Sportiello for asking good questions. We also thank the anonymous referees for numerous improvements on the presentation of the paper.


\section{Background and definitions}\label{S:Back}

\subsection{Multi-type TASEP on a ring} Consider a vector $\mbf {m}=(m_1, \dots, m_n)$ of positive integers and a ring with $N$ sites (we assume that $N \ge \sum m_i$), labeled $0,\dots,N-1 \pmod{N}$ from left to right starting at some particular site $0$. A state of the $\mbf{m}$-TASEP chain is an assignment of $\sum_im_i$ labeled particles to positions on the ring such that no two particles are on the same site.
Exactly $m_i$ of the particles are labeled $i$.
The dynamics of the chain is defined as follows. A particle is chosen uniformly at random and if the site to left of the particle is empty or contains a particle with larger label then the two sites swap. 
\[\underline{\phantom{i} }\ \underline{i} \to \underline{i}\ \underline{\phantom{i} }\quad \text{ and }\quad 
\underline{k}\ \underline{i} \to \underline{i}\ \underline{k} \quad \text{ if $k>i$ }\ \]
(Thus we can think of vacancies as particles labeled $n+1$, but the notation becomes simpler by not doing so.)
It is important to note that all particles less than $i$ "look the same" to $i$ (as do all larger than $i$). This observation is called the projection principle. So, for instance, if one is interested only in the behavior of class $i$ particles, one can instead study the chain where all particles less than $i$ have the same class $1$, all class $i$ particles have class $2$, and the rest class $3$. 

Exclusion processes have been studied extensively in general, and this $\mbf{m}$-TASEP has been considered by several authors \cite{AS, Angel, AL2, EFM, FM1, FM2, Lam, LM}. Both Matrix Ansatz solutions and more combinatorial solutions have been suggested. 

Let $\SD_\mbf{m}(N)$ be the stationary distribution of the TASEP. We define $\SD_n$ as the limit of $\SD_\mbf{m}$ when $\mbf{m} = (\underbrace{1, \dots, 1}_n)$ is fixed and $N$ tends to infinity, while scaling the ring to have length 1 (this will be made precise below). Note that we define a limit of stationary distributions and not the limit of the TASEP itself. We leave that as an interesting challenge.


\subsection{Multiline queues}

We will make extensive use of \emph{multiline queues} (MLQ's), originally defined by Ferrari and Martin \cite{FM2}. We distinguish between \emph{discrete} MLQ's and \emph{continuous} MLQ's.

A {\it discrete MLQ} of type $\mbf{m} = (m_1, \dots, m_n)$ is an $n \times N$ array, with $m_1 + \dots +m_i$ boxes in row $i$ for $1 \leq i \leq n$. We label the rows $1, \dots, n$ from top to bottom.
Given such an array, there is a labeling procedure which assigns a label to each box. See Figure \ref{fi_1} for an example. We label the boxes row by row from top to bottom. Suppose we have just labeled row $i$. Pick any order of the boxes in row $i$ such that boxes with smaller label come before boxes with larger label. Now go through the boxes in this order. When considering a box labeled $k$, find the first unlabeled box in row $(i+1)$, going weakly to the right (cyclically) from the column of the box, and label that box $k$. When this is done, some boxes (in total $m_{i+1}$) remain unlabeled in row $i+1$. Label these $i+1$. Thus all boxes in the first row are labeled $1$.
By {\it $k$-bully path} we mean the path of a label $k$ from its starting position in row $k$ directly down and then along the $k+1$st row to its box with label $k$, and so forth all the way down to the bottom row. If two $k$-bully-paths are arriving at the same label $k$ it is not well defined which one turns downwards and which one continues on the same row, but it will not matter for our purposes.

Note that we can alternatively label the boxes by finding the $k$-bully paths for $k = 1, \dots, n$ in order. For example, if $m_i = 1$, the $1$-bully path is obtained by always taking the next box weakly cyclically to the right of the current box, starting with the unique box on row $1$.

A {\it continuous MLQ} of type $\mbf{m}$ is a sequence of $n$ "continuous" rows with $m_1+\dots+m_i$ boxes in row $i$. In this case we consider the location of the boxes to be numbers in the continuous interval $[0,1)$. The location of the boxes are chosen uniformly at random in each row such that two boxes cannot have the exact same position. We label the boxes using the same labeling procedure as for discrete MLQ's. 

The distribution on the last row (in fact, on any row) in a continuous MLQ is the limit of the corresponding distribution for discrete MLQ's. We will use this in the proofs in Section \ref{S:Densities}.

\begin{figure}
\begin{tikzpicture}[scale=1]
\matrix [column sep=0.1cm, row sep = 0.1cm, ampersand replacement=\&] 
{
\node {$\vac$}; \& \node{$\vac$}; \& \node(a1) {$\occ_{\ \,}$}; \& \node {$\vac$}; \& \node {$\vac$}; \& \node {$\vac$}; \& 
\node(a21) {$\vac$}; \& \node {$\vac$};  \& \node {$\vac$}; \\
\node {$\vac$}; \& \node {$\vac$}; \& \node(a2) {$\vac$}; \& \node {$\vac$}; \& \node {$\vac$}; \& \node(a3) {$\occ_{\ \,}$}; \& 
\node(a22) {$\occ_{\ \,}$}; \& \node(b1) {$\occ_{\ \,}$}; \& \node {$\vac$}; \\
\node(b3) {$\occ_{\ \,}$}; \& \node(c1){$\occ_{\ \,}$}; \& \node {$\vac$}; \& \node(c21) {$\occ_{\ \,}$};\& \node {$\vac$}; \& 
\node{$\vac$}; \& \node(a23)  {$\occ_{\ \,}$};\& \node(b2) {$\vac$}; \& \node(a24) {$\occ_{\ \,}$};\\ 

\node(a26) {$\occ_{\ \,}$}; \& \node(b4){$\vac$}; \& \node {$\occ_{\ \,}$}; \& \node(c3) {$\occ_{\ \,}$}; \&\node(c23) {$\occ_{\ \,}$}; \&
\node {$\occ_{\ \,}$};\& \node(d1) {$\occ_{\ \,}$}; \& \node(d21) {$\occ_{\ \,}$}; \& \node(a25) {$\vac$};\\
\node(a27) {$\occ_{\ \,}$}; \& \node(b5){$\occ_{\ \,}$}; \& \node {$\occ_{\ \,}$}; \& \node(c4) {$\occ_{\ \,}$}; \& \node(c24) {$\occ_{\ \,}$};\&
\node(a4) {$\occ_{\ \,}$};\& \node(d2) {$\occ_{\ \,}$}; \& \node(d22) {$\occ_{\ \,}$}; \& \node {$\occ_{\ \,}$};\\
};
\end{tikzpicture}
\caption{A discrete multiline queue with ${\bf m}=(1,2,2,2,2)$.
}
\label{fi_1}
\end{figure}
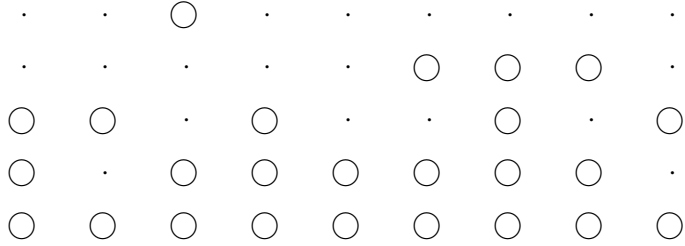


\begin{figure}
\begin{tikzpicture}[scale=1]
\matrix [column sep=0.1cm, row sep = 0.1cm, ampersand replacement=\&] 
{
\node {$\vac$}; \& \node{$\vac$}; \& \node(a1) {$\occ_{1}$}; \& \node {$\vac$}; \& \node {$\vac$}; \& \node {$\vac$}; \& 
\node(a21) {$\vac$}; \& \node {$\vac$};  \& \node {$\vac$}; \\
\node {$\vac$}; \& \node {$\vac$}; \& \node(a2) {$\vac$}; \& \node {$\vac$}; \& \node {$\vac$}; \& \node(a3) {$\occ_1$}; \& 
\node(a22) {$\occ_2$}; \& \node(b1) {$\occ_2$}; \& \node {$\vac$}; \\
\node(b3) {$\occ_2$}; \& \node(c1){$\occ_3$}; \& \node {$\vac$}; \& \node(c21) {$\occ_3$};\& \node {$\vac$}; \& \node(a12) {$\vac$};\& 
\node(a23) {$\occ_1$}; \& \node(b2) {$\vac$}; \& \node(a24) {$\occ_2$};\\ 

\node(a26) {$\occ_2$}; \& \node(c12) {$\phantom{\vac}$}; \& \node(b4){$\occ_2$}; \& \node(c3) {$\occ_3$}; \&\node(c23) {$\occ_3$}; \&\node(d1)  {$\occ_4$};\& 
\node{$\occ_1$}; \& \node(d21) {$\occ_4$}; \& \node(a25) {$\vac$};\\
\node(a27) {$\occ_2$}; \& \node{$\occ_5$}; \& \node(b5) {$\occ_{2}$}; \& \node(c4) {$\occ_3$}; \& \node(c24) {$\occ_3$};\&\node(d2)  {$\occ_4$};\& 
\node(a4){$\occ_1$}; \& \node(d22) {$\occ_4$}; \& \node {$\occ_{5}$};\\ 
};

\draw [-,thick,red] (a1.center) -- (a2.center) -- (a3.center) -- (a12.center)--(a23.center)--(a4.center);
\draw[-,thick,blue] (a22.center) -- (a23.center) -- (a24.center) -- (a25.center)--(5,-0.7); 
\draw [-,thick, blue] (-5,-0.75) -- ([yshift=-1pt] b4.center) -- (b5.center);
\draw [-,thick, blue] (b1.center) -- (b2.center) -- (5,0.05);
\draw [-,thick, blue] (-5,0) -- (b3.center)  -- (a27.center);
\draw [-,thick,green] (c1.center) -- ([yshift=1pt] c12.center) -- ([yshift=1pt] c23.center) -- (c24.center);
\draw [-,thick,green] (c21.center) -- (c4.center);
\draw [-,thick] (d1.center) -- (d2.center);
\draw [-,thick] (d21.center) -- (d22.center);
\end{tikzpicture}
\caption{The bully paths and the labeling of the multiline queue in Figure \ref{fi_1}.}
\label{fi_2}
\end{figure}
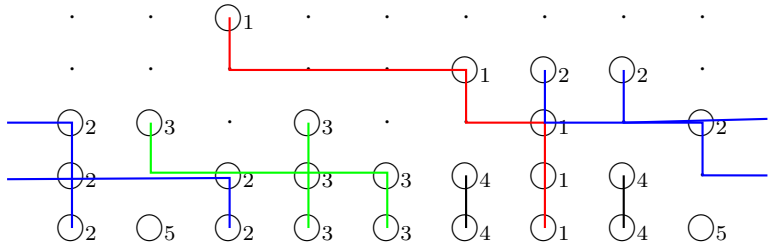

%

\begin{figure}
\begin{tikzpicture} 
\node[circle,draw=black] at (2.3,1){3}; 
\node[circle,draw=black] at (4.7,1){1}; 
\node[circle,draw=black] at (6.8,1) {2};
\node[circle,draw=black] at (8.0,1) {1};
\draw[help lines](0.5,1.5)--(8.5,1.5);
\node[circle,draw=black] at (3.0,2) {1};
\node[circle,draw=black] at (3.6,2){2};
\node[circle,draw=black] at (7.2,2){1};
\draw[help lines](0.5,2.5)--(8.5,2.5);
\node[circle,draw=black] at (2.7,3){1};
\node[circle,draw=black] at (5.1,3){1};
\draw [-,thick,red] (2.7,3)--(2.7,2)--(3.0,2)--(3.0,0.95)--(4.7,0.95);
\draw [-,thick,red] (5.1,3)--(5.1,2)--(7.2,2)--(7.2,1)--(8.0,1);
\draw [-,thick,blue] (3.6,2)--(3.6,1)--(6.8,1);
\end{tikzpicture}
\caption{A continuous multiline queue.}
\label{fi_3}
\end{figure}
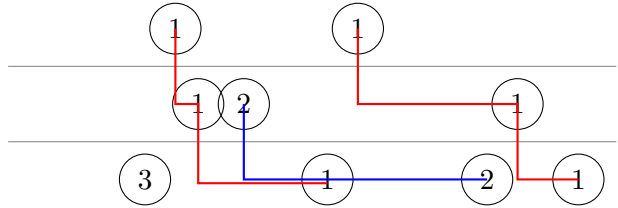

The motivation for these definitions is the following theorem.

\begin{theorem}[Ferrari-Martin]
\label{th_FM}
Let $X$ be an $\mbf{m}$-TASEP distributed word. Then at stationarity
\[
\mbb{P}[X = u] = \frac{n_u}{\prod_{i=1} ^r \binom{n}{m_1 + \dots + m_i}},
\]
where $n_u$ is the number of $\mbf{m}$-MLQ's whose bottom row is labeled $u$.
\end{theorem}

From this theorem it follows that $\SD_n$ is the distribution of the labels on the bottom row for a uniformly chosen continuous MLQ. 


\subsection{The process of the last row}
\label{S:last_row}

Theorem \ref{th_FM} can easily be extended to the case where $m_n = 0$, as we now explain. Consider a $(m_1, \dots, m_{n-1}, 0)$-MLQ. By the theorem, the labeling of row $n-1$ has TASEP distribution. Using the proof of \cite{FM2}, it is easy to show that row $n$ also has TASEP distribution, \emph{of the same type}. It follows that the probabilistic map from the $(n-1)$st row to the $n$th row represents another Markov chain with the same distribution! Here is an example of how we will use this fact. Consider Figure \ref{fi_polr}. In the top row we have sampled a word from the TASEP distribution, and on the second row we have selected 4 positions for the boxes, uniformly at random. Then we use the same labeling procedure as before. The claim, then, is that the bottom row also has TASEP distribution. 

Formally, the discrete process of the last row is a Markov chain where the states are assignments of labeled particles to positions, $m_i$ particles labeled $i$ for each $i$, on the circle $Z_{N}$. The transitions are uniformly random assignments of positions of $\sum_i m_i$ boxes on a row below and they are labeled with the same labeling procedure as for multiline queues. The unique stationary distribution of this process thus coincides with the stationary distribution of the $\mbf{m}$-TASEP and the bottom row of a uniformly random $\mbf{m}$-MLQ. 

The continuous process of the last row is then obtained by letting $N\to \infty$ (while scaling the circle to unit length), which similarly gives the same stationary distribution as the continuous MLQ of type $\bf m$.

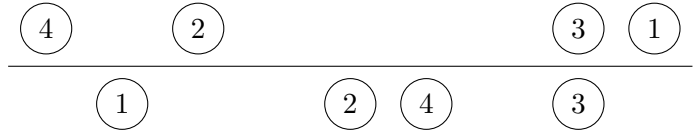
\begin{figure}
\begin{tikzpicture} 
\node[circle,draw=black] at (1,2) {4}; 
\node[circle,draw=black] at (3,2) {2}; 
\node[circle,draw=black] at (8,2) {3};
\node[circle,draw=black] at (9,2) {1};
\draw(0.5,1.5)--(9.5,1.5);
\node[circle,draw=black] at (2,1) {1};
\node[circle,draw=black] at (5,1) {2};
\node[circle,draw=black] at (6,1) {4};
\node[circle,draw=black] at (8,1) {3};
\end{tikzpicture}
\caption{The bottom row is the output of the process of the last row. The top row together with the positions of the boxes in the second row are given.}
\label{fi_polr}
\end{figure}

\section{Discrete and continuous density functions}\label{S:Densities}

Recall that in the definition of $\SD_n$ we have taken $\mbf{m}=(1,\dots,1)$, so the particles form a permutation. We now define the key quantities of interest.

\begin{enumerate}
\item [(a)] For a permutation $\pi$, the number $G_\pi (b_1, \dots, b_n; N)$ of discrete MLQ's of length $N$ such that the labels of the boxes in the bottom row are $\pi_1, \dots, \pi_n$, at positions $b_1 < \dots < b_n$. For fixed $N$, summing $G_\pi (b_1, \dots, b_n; N)$ over all permutations $\pi$ and all increasing sequences $b_1 < \dots < b_n$, we get $Z_N = \binom N1\binom N2 \dots \binom Nn$, the total number of discrete MLQ's.

\item [(b)] We get the corresponding continuous probability density function for $0<q_1<\dots<q_n<1$ as a limit:
\[
g_\pi(q_1, \dots, q_n) = 
\lim_{\delta\to 0} \delta^{-n}
\lim_{N\to\infty} \frac{1}{Z_N} 
\sum_{(x_1,\dots,x_n) \in  \mathbb{Z}^n \cap \prod_i [q_iN, (q_i+\delta)N)}
G_\pi(x_1, \dots, x_n; N).
\]


\item [(c)] The probability $p_\pi = \int_{0 < q_1 < \dots < q_n < 1}  g_\pi(q_1, \dots, q_n) dq_1\dots dq_n$ that the letters at the bottom of a random continuous MLQ form the permutation $\pi$.\\
\end{enumerate}
\begin{remark}\label{R:count}
The probability $p_\pi$ can also be obtained by a finite computation as follows. With probability $1$, the horizontal positions of all boxes in the 
MLQ are unique, and can hence be ordered from right to left, by numbering them from $1$ to $\binom{n}{2}$. Such a numbering we call a {\it 
placement} $(a_{ij})$ where $a_{ij}$ is the label given to the $j$th box from the right in the $i$th row (from the top, as always). Since we always 
have $a_{i(j+1)} > a_{ij}$, clearly there are ${\binom{\binom{n+1}{2}}{1,\dots,n}}$ unique placements for a $(1, \dots, 1)$-MLQ. The placement 
determines the labeling of the bottom row. Therefore, we have $p_\pi = \frac{k}{\binom{\binom{n+1}{2}}{1,\dots,n}}$, where $k$ is the number of 
placements for which the bottom row is labeled $\pi$.

A refined version of this observation allows for computing the polynomials $g_\pi$. We have used this for computing examples, but not in the proofs.
\end{remark}

\subsection{Probability density functions for the continuous chain}\label{S:ProbDensity}
As an example, consider the case when $n=2$. Which permutation appears depends solely on whether or not the box in row 1 is between the two boxes in row 2 or not. This gives us directly $g_{21}(q_1,q_2)=q_2-q_1$ and $g_{12}(q_1,q_2)=1-(q_2-q_1)$.
The complexity of this direct approach grows very quickly but we have been able to compute it using a computer for all permutations when $n\le 5$. For some special classes of permutations $\pi$, we have been able to use the results on the discrete chain, see Section \ref{S:Disc chain}, to understand the probability density function $g_\pi$ for the continuous distribution $\SD_n$ when the particles form the permutation $\pi$. 


Our first result is a formula for $p_{w_0}$ where $w_0 = n(n-1)\dots 1$ is the reverse permutation.

\begin{theorem} \label{T:w0}
For any $n\ge 2$ and $0\le q_1<\dots<q_n<1$ we have 
\[
	g_{w_0}(q_1,\dots,q_n)=n! \prod_{1\le k<l\le n} (q_l-q_k)
\]
\end{theorem}
\begin{proof}
By setting $q_i=b_i/N$ in the formula in definition (b) and letting $N\to\infty$ in \refP{P:w0}, the result follows with an easy calculation.
\end{proof}

We note that a similar formula is proven in \cite{AAV} for the TASEP speed process, namely for the distribution function of the "TASEP speeds" $(U_1, \dots, U_n)$ conditioned on $U_1 > \dots > U_n$ (so that particles $1, \dots, n$ never cross). We cannot see any closer similarity between our process and the speed process than that in these two particular cases, the computation boils down to counting the same type of MLQ's (which in turn are closely related to Gelfand-Tsetlin patterns).

Next, we study the permutations $s_kw_0=n\dots (k+2) k (k+1) (k-1)\dots 21$, where the numbers $k$ and $k+1$ have switched places in $w_0$. 
We have the following result for $k=1,2$, which was surprisingly difficult to prove.
As above, the proof follows by setting $q_i=b_i/N$ and letting $N\to\infty$ in a corresponding statement for the discrete chain \refT{T:k=1,2}.

\begin{theorem}\label{T:Cont12}
For any $0\le q_1<\dots<q_n<1$, we have 
\begin{align*}
	g_{s_1w_0} &= \left(\frac{\partial}{ \partial q_n} - 1 \right) g_{w_0}, \quad &\text {for $n\ge 2$}\\
	g_{s_2w_0} &= \left(\frac{1}{2}\frac{\partial^2}{\partial q_{n-1}\partial q_n} - 1 \right) g_{w_0},\quad &\text {for $n\ge 3$}.
\end{align*}
\end{theorem}
Our computational data for $n\le 5$ suggest the following for any $k\ge 1$.

\begin{conjecture}\label{Conj:Pdfoneaway}
For any $n>k\ge 1$ and any $0\le q_1<\dots<q_n<1$, we have 
\[
	g_{s_kw_0} = \left(\frac{1}{k!}\frac{\partial^k}{\partial q_{n-k+1}\dots \partial q_n} - 1 \right) g_{w_0}.
\]
\end{conjecture}

\smallskip
\noindent
\begin{example} For $n=4$, we know the following to be true 
\[
g_{4321} = 4!\prod_{1\leq i < j \leq 4} (q_j - q_i),
\qquad
g_{4312} =  \left(\frac{\partial}{\partial q_4} - 1 \right) g_{4321}
\]

\[
g_{4231} = \left(\frac{1}{2} \frac{\partial^2}{\partial q_3 \partial q_4} - 1\right) g_{4321},
\qquad
g_{3421} = \left(\frac{1}{6} \frac{\partial^3}{\partial q_2 \partial q_3 \partial q_4} - 1 \right) g_{4321}.
\]
\end{example}
\smallskip

The polynomials $g_w$ satisfy some interesting relations whose general form we have not been able to pin down exactly. For example, for $n \leq 4$, all the $g_w(x_1, \dots, x_n)$ satisfy Laplace's equation
\[
\frac{\partial^2 g_w}{\partial x_1^2} + \frac{\partial^2 g_w}{\partial x_2^2} + \dots + \frac{\partial^2 g_w}{\partial x_n^2} = 0.
\]

It's a classical fact \cite{ABW} that any such {\it harmonic polynomial} can be expressed as a linear combination of partial derivatives of a Vandermonde determinant (the converse, that any such combination is harmonic, is immediate). Since $g_{w_0}$ is a Vandermonde determinant, it follows that for each $w$, there is \emph{some} linear combination of its partial derivatives whose value is $g_w$. In each case there seems to be a particularly simple one. Here are some examples:

\[
g_{132} = \left( 1+ \frac{\partial}{\partial q_1} + \frac{1}{2}\frac{\partial^2}{\partial q_1^2} \right)g_{321}
\]
\[
g_{1432} = \left( -1- \frac{\partial}{\partial q_1} - \frac{1}{2}\frac{\partial^2}{\partial q_1^2} - \frac{1}{6}\frac{\partial^3}{\partial q_1^3} \right)g_{4321}
\]
\[
g_{4132} = \left(1 - \frac{\partial}{\partial q_3} - \frac{\partial}{\partial q_4} +\frac{1}{2} \frac{\partial^2}{\partial q_3 \partial q_4}\right) g_{4321}
\]

\[
g_{4213} = \left( 1- \frac{\partial}{\partial q_4} + \frac{1}{2}\frac{\partial^2}{\partial q_4^2} \right)g_{4321}
\]

For $n = 5$, we have found that only $15$ of the $24$ (up to cyclic shifts) polynomials $g_w$ satisfy Laplace's equation. 

Since $g_\pi(q_1, \dots, q_n) = g_\pi(q_1+t,\dots,q_n+t)$ for infinitely many $t$, this holds on the level of polynomials. Using this, it is easy to check that, as polynomials, $g_{\pi_1 \dots \pi_n} (q_1, \dots, q_n) = g_{\pi_2 \dots \pi_n \pi_1}(q_2, q_3, \dots, q_n, 1+q_1)$.
Hence it suffices to check that the Laplacian vanishes on a candidate from each cyclic class to conclude that it vanishes on all of them.
Thus, there cannot be any linear recursion using differentiation operators between the polynomials $g_w$ in general. 

\noindent
{\bf Open problem:} Is there some other set of operators which coincides with differentiation for small $n$ and does extend the pattern above to larger $n$?
\smallskip

The reader may note that the part of maximal degree in $g_u$ appears to be $\pm g_{w_0}$, where $w_0 = 4321$, and we choose $+$ if and only if $\ell(w_0)-\ell(u)$ is even.

We end this section with a more general conjecture. We reached this conjecture by contemplating the proof of \refT{T:k=1,2}. We have checked it for $n\leq 5$.

\begin{conjecture}\label{Conj:Pdfmanyaway} For $\mbf{k}$ such that $n>k_1>k_2+1>k_3+2>\dots >k_r+r-1>r-1$ 
and any $0\le q_1<\dots<q_n<1$,
\[
g_{s_{k_1}\dots s_{k_r}w_0} = \left(\frac{1}{k_r!}\frac{\partial^{k_r}}{\partial q_{n-k_r+1}\dots \partial q_n} - 1 \right) g_{s_{k_1}\dots s_{k_{r-1}}w_0}.
\]
\end{conjecture}

An example, where the conjecture is true, is $n=4, k_1=3, k_2=1$:
\[
g_{3412} = \left(\frac{1}{6}\frac{\partial^3}{\partial q_2\partial q_3\partial q_4} - 1 \right)\left(\frac{\partial}{\partial q_4} - 1\right) g_{4321}.
\]

\subsection{Probabilities of the discrete chain}\label{S:Disc chain}
In this section, we will prove some exact formulas for $G_\pi$, which are used to establish the probability density functions $g_\pi$ in Section \ref{S:ProbDensity}.
We start with the easiest case, which is the reverse permutation $w_0 = n(n-1)\dots 1$. We say that a bully path \emph{wraps} when it moves from the right side to the left, that is, from $(i,N-1)$ to $(i,0)$. 

\begin{proposition} \label{P:w0}
For any $N\ge n\ge 2$ we have 
\[
  G_{w_0}(b_1,\dots,b_n;N)=\det \left[ \binom{b_i+j-1}{j-1} \right] _{1\le i,j\le n}=\prod_{1\le k<l\le n} (b_l-b_k)  \prod_{d=1}^{n-1} \frac{1}{d!}.
\]
\end{proposition}

\begin{proof}

Suppose we have a discrete MLQ whose bottom row $n$ is labeled by the reverse permutation, with a particle labeled $n+1-i$ at position $b_i$, for $1\leq i\leq n$, where $b_1<\dots<b_n$. It is a direct consequence of the construction of MLQs that the positions of the boxes in row $n-1$, $b'_1<\dots<b'_{n-1}$ must be such that $b_i<b'_i\le b_{i+1}$ for $1\le i<n$, hence they must also correspond to the reverse permutation (of length $n-1$). It follows by induction that each row is labeled by a reverse permutation.
 
Thus the bully paths do not wrap and are non-intersecting, so to enumerate them we may use the Lindstr\"om-Gessel-Viennot lemma, see for example \cite [Chapter 2.7]{EC1}. 
Extend each bully path to start at the beginning of the row, that is at positions $(r,0)$ for $1\le r\le n$ (matrix notation). See Figure \ref{fi_LGV_simple} for an illustration.

The number of (non-wrapping) paths using only right and down steps from $(r,0)$ to $(n,b_i)$ is $\binom{b_i+n-r}{n-r}$. Setting $j=n+1-r$ gives the determinant in the proposition. The factor $\prod_{d=1}^{n-1} \frac{1}{d!}$ can be taken out and using column operations we reduce to the standard form of the Vandermonde determinant.
\end{proof}

\begin{figure}[htb]
\begin{tikzpicture}
\matrix [column sep=0cm, row sep = 0.15cm] 
{
\node[left] {$(1,0)$}; \node(a1)  {$\bullet$}; &  &  & & & \node(a2) {\ $\circ_1$}; & & & &\\
\node[left] {$(2,0)$};\node(b1) {$\bullet$}; & & & &\node(c2) {\ $\circ_2$};  & \node(a3){}; & \node(b2) {\ $\circ_1$};  & & &  & \\
\node[left] {$(3,0)$};\node(d1) {$\bullet$}; & & &\node(d2){$\circ_3$}; &\node(c3){}; & \node(c4) {\ $\circ_2$}; & \node(b3) {};&  \node(b4) {\ $\circ_1$}; & & \\
& & &\node(d3){}; & \node(d4){\ $\circ_3$}; & \node(c5) {}; & \node(c6) {\ $\circ_2$}; &\node(b8){}; & \node(b6) {\ $\circ_1$};& \\
\node {$\vdots$};& & & & &\node[green] {\hskip-20pt$\ddots$};   & &  &\node[blue]{$\hskip-30pt\ddots$};  & \node[red] {\hskip-30pt $\ddots$}; \\
\node[left] {$(n,0)$};\node(e1) {$\bullet$}; & \node(e2){$\times$}; & &\node{$\dots$}; & &    \node(d7) {$\times$};    & & & \node(c7) {$\times$}; &  \node(b7) {$\times$}; \\
&\node {$(n,b_1)$};& & & & \node{\hskip5pt $(n,b_{n-2})$};    & & &\node {\hskip5pt$(n,b_{n-1})$}; & \node {\hskip5pt$(n,b_n)$}; \\
}; 
\draw [-,red] (a1.center) -- (a2.center) -- (a3.center) -- (b2.center) -- (b3.center) -- (b4.center) -- (b8.center) -- (b6.center) --([yshift=-15pt] b6.center);
\draw [-,blue] (b1.center) -- (c2.center) -- (c3.center) -- (c4.center) -- (c5.center) -- (c6.center)--([yshift=-15pt] c6.center);
\draw [-,green] (d1.center) -- (d2.center) -- (d3.center) --(d4.center)--([yshift=-15pt] d4.center);
\draw [-,black] (e1.center) -- (e2.center);
\draw [-,green] (d7.center) -- ([xshift=-5pt] d7.center) --([xshift=-5pt, yshift=5pt] d7.center);
\draw [-,blue] (c7.center) -- ([xshift=-5pt] c7.center) --([xshift=-5pt, yshift=5pt] c7.center);
\draw [-,red] (b7.center) -- ([xshift=-5pt] b7.center) --([xshift=-5pt, yshift=5pt] b7.center);
\end{tikzpicture}
\caption{Counting multiline queues with lattice paths.}
\label{fi_LGV_simple}
\end{figure}
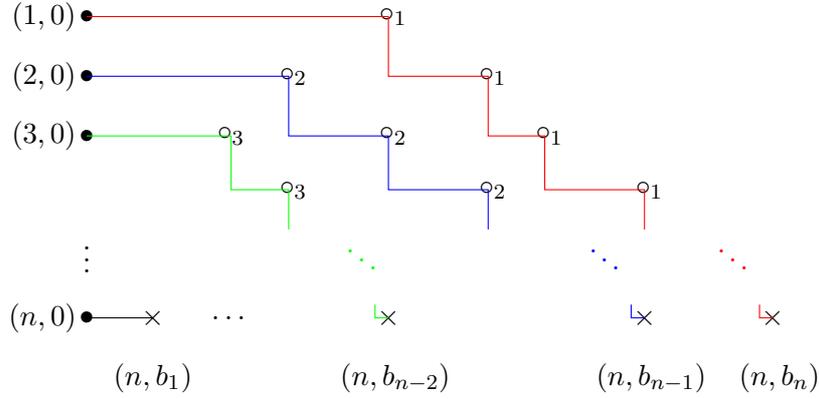

Our next task is to prove the statement for the discrete MLQ's used in the proof of \refT{T:Cont12}.
For $1\le k\le n$, let $A_k$ be the matrix with entries $\binom{b_i+j-1}{j-1}$ in rows $1\le i\le n-k$ and entries 
$\binom{b_i+j-2}{j-2}$ in rows $n-k<i \le n$. Also, let $P_i$ be the bully path of the $i$'th class particle extended so it starts at $(i,0)$.

\begin{theorem}\label{T:k=1,2} For $k=1,2$, we have
\[G_{s_kw_0}=\binom{N}{k}\det A_k - G_{w_0}. 
\]
\end{theorem}

\begin{proof}
We will first consider the case $k=1$. 
We distinguish between two different types of MLQ's that can result in the permutation $s_1w_0$, depending on whether or not
$P_1$ is wrapping.

The first type is an MLQ $\con$ where there is no bully path wrapping. Then $P_1$ and $P_2$ will touch at some point $(r,c)$, after which $P_1$ 
will be below or on $P_2$ (bully paths may, however, coincide only in points and horisontal segments), see Figure \ref{fi_typeI}.
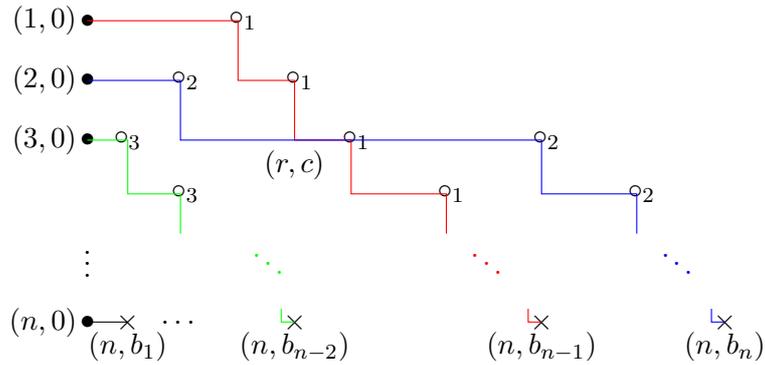
\begin{figure}[htb]
\begin{tikzpicture}
\matrix [column sep=0cm, row sep = 0.15cm] 
{
\node[left] {$(1,0)$}; \node(a1)  {$\bullet$};    & & & \node(a2) {\ $\circ_1$}; & & & &\\
\node[left] {$(2,0)$};\node(b1) {$\bullet$};  & &\node(c2) {\ $\circ_2$};  & \node(a3){}; & \node(b2) {\ $\circ_1$};  & & &  & \\
\node[left] {$(3,0)$};\node(d1) {$\bullet$};  &\node(d2){$\circ_3$}; &\node(c3){}; & &\node(b3) {};& \node(c4) {\ $\circ_1$}; & & \node(b4) {\ $\circ_2$};  \\
& \node(d3){}; & \node(d4){\ $\circ_3$};   && & \node(c5) {}; & \node(c6) {\ $\circ_1$}; &\node(b8){}; & \node(b6) {\ $\circ_2$}; \\
\node {$\vdots$};& &  & &\node[green] {\hskip-20pt$\ddots$};   &  &\node[red]{$\hskip30pt\ddots$};  && \node[blue] {\hskip30pt $\ddots$}; \\
\node[left] {$(n,0)$};\node(e1) {$\bullet$}; & \node(e2){$\times$}; & \node{$\dots$}; & &    \node(d7) {$\times$};    & & & \node(c7) {$\times$}; & &  \node(b7) {$\times$}; \\
}; 
\draw [-,red] (a1.center) -- (a2.center) -- (a3.center) -- (b2.center) -- (b3.center) -- (c4.center)-- (c5.center) -- (c6.center)--([yshift=-15pt] c6.center);
\draw [-,blue] (b1.center) -- (c2.center) -- (c3.center) -- (b4.center) -- (b8.center) -- (b6.center) --([yshift=-15pt] b6.center) ;
\draw [-,green] (d1.center) -- (d2.center) -- (d3.center) --(d4.center)--([yshift=-15pt] d4.center);
\draw [-,black] (e1.center) -- (e2.center);
\draw [-,green] (d7.center) -- ([xshift=-5pt] d7.center) --([xshift=-5pt, yshift=5pt] d7.center);
\draw [-,red] (c7.center) -- ([xshift=-5pt] c7.center) --([xshift=-5pt, yshift=5pt] c7.center);
\draw [-,blue] (b7.center) -- ([xshift=-5pt] b7.center) --([xshift=-5pt, yshift=5pt] b7.center);
\node [below] at (b3) {$(r,c)$};
\node [below] at (e2) {$(n,b_1)$};
\node [below] at (d7) {$(n,b_{n-2})$};
\node [below] at (c7) {$(n,b_{n-1})$};
\node [below] at (b7) {$(n,b_n)$};
\end{tikzpicture}
\caption{Paths of type I for the case $k=1$.}
\label{fi_typeI}
\end{figure}

We could also describe this as the reverse permutations on the first $r-1$ rows of $\con$, after which 
the 1 and 2 switch places. 

To count MLQ's of the first type, we define an injection into sets of certain non-intersecting paths $\mathcal L_I=\{L_n,\dots, L_1\}$. The path $L_1$ is formed by concatenating $P_1$ from $(1,0)$ to $(r,c)$ with $P_2$ from $(r,c)$ to $(n,b_n)$ and then lifting the 
resulting path one step upwards. So, $L_1$ is a path from $(0,0)$ to $(n-1,b_n)$ that passes through $(r-1,c)$. The path $L_2$ is formed by 
concatenating $P_2$ from $(2,0)$ to $(r,c)$ with $P_1$ from $(r,c)$ to $(n,b_{n-1})$. For 
$3\le i \le n$, $L_i=P_i$. Thus $\mathcal L_I$ is a set of non-intersecting lattice paths with starting positions
$\mathcal S_I=\{(n,0),\dots,(3,0),(2,0),(0,0)\}$ and ending positions $\mathcal M_I=\{(n,b_1),\dots,(n,b_{n-2}),(n,b_{n-1}),(n-1,b_n)\}$. 
By the Lindstr\"om-Gessel-Viennot lemma, all such sets of paths are counted by the determinant of the matrix:
\[
D_I= \left[\begin{array}{ccc|c}
\ddots&\vdots &\iddots &\vdots\\
\cdots &\binom{b_i+j-1}{j-1}&\cdots &\binom{b_i+n}{n}\\
\iddots & \vdots  &\ddots &\vdots \\
\hline
\cdots &\binom{b_n+j-2}{j-2}&\cdots &\binom{b_n+n-1}{n-1}\\
\end{array}\right].
\]

Sets of paths $\mathcal L_I$ where the vertical distance between $L_1$ and $L_2$ is always two or more do not come from an MLQ $\con$ of the first type.  
Since if we lower $L_1$ one vertical step they would still not intersect.
To subtract this over-count, we may thus count the number of non-intersecting paths from $\mathcal S=\{(n,0),\dots,(2,0),(1,0)\}$ to $\mathcal M=\{(n,b_1),
\dots,(n,b_{n-1}),(n,b_n)\}$, which is precisely $G_{w_0}$ by Proposition \ref{P:w0}. Thus the number of MLQ's of the type I is $\det(D_1) -
G_{w_0}$.

The second type of MLQ's $\con$, has a bully path wrapping and the only possibility is that $P_1$ wraps in row 2 before finding a box to label 1, see Figure \ref{F:k1}. Thus $P_1$ will overlap with $P_2$ in the beginning of row two and possibly more later. Apart from that, no paths are touching. We now describe a bijection from such MLQ's to non-intersecting lattice paths $\mathcal L_{I\! I}=\{L_n,\dots, L_{1}, L_{0}\}$ with starting positions $\mathcal S_{I\! I}=\{(n,0),\dots,(2,0),(1,0),(0,0)\}$ and ending positions $\mathcal M_{I\! I}=\{(n,b_1),\dots,(n,b_{n-2}),(n,b_{n-1}),(n-1,b_n),(1,N-1)\}$.
We define $L_0$ as a translation one step upwards of $P_1$ from $(1,0)$ to $(2,N-1)$.
The path $L_1$ is a translation one step upwards of $P_2$.  $L_2$ is the part of $P_1$ going from $(2,0)$ to $(n,b_{n-1})$. For $3\le i \le n$, $L_i=P_i$.

\begin{figure}[h!]
\begin{tikzpicture}
\matrix [column sep=0cm, row sep = 0.15cm] 
{
\node(a1) {$\bullet$}; &  &  & & & &  & \node(a2) {\ $\circ_1$}; & &\\
\node(b1) {$\bullet$}; & & & &\node(c2) {\ $\circ_1$};  & & \node(b2) {\ $\circ_2$};  & \node(a3){}; & &  &\node(a4) {$\times$};  \\
\node(d1) {$\bullet$}; & & &\node(d2){$\circ_3$}; &\node(c3){}; & \node(c4) {\ $\circ_1$}; & \node(b3) {};&  \node(b4) {\ $\circ_2$}; & & \\
& & &\node(d3){}; & \node(d4){\ $\circ_3$}; & \node(c5) {}; & &  \node(c6) {\ $\circ_1$}; & \node(b6) {\ $\circ_2$};& \\
\node {$\vdots$};& & & & &\node[green] {\hskip-20pt$\ddots$};   & &  &\node[red]{$\hskip-30pt\ddots$};  & \node[blue] {\hskip-30pt $\ddots$}; \\
\node(e1) {$\bullet$}; & \node(e2){$\times$}; & &\node{$\dots$}; & &    \node(d7) {$\times$};    & & & \node(c7) {$\times$}; &  \node(b7) {$\times$}; \\
&\node {$b_1$};& & & & \node{\hskip5pt $b_{n-2}$};    & & &\node {\hskip5pt$b_{n-1}$}; & \node {\hskip5pt$b_n$}; \\
}; 
\draw [-,red] (a1.center) -- (a2.center) -- (a3.center) -- (a4.center);
\draw [-,blue] ([yshift=2pt] b1.center) -- ([yshift=2pt] b2.center) -- (b3.center) -- (b4.center) -- (c6.center) -- (b6.center)--([yshift=-15pt] b6.center);
\draw [-,red] ([yshift=-1pt] b1.center) -- ([yshift=-1pt] c2.center) -- (c3.center) -- (c4.center) -- (c5.center) -- (c6.center)--([yshift=-15pt] c6.center);
\draw [-,green] (d1.center) -- (d2.center) -- (d3.center) --(d4.center)--([yshift=-15pt] d4.center);
\draw [-,black] (e1.center) -- (e2.center);
\draw [-,green] (d7.center) -- ([xshift=-5pt] d7.center) --([xshift=-5pt, yshift=5pt] d7.center);
\draw [-,red] (c7.center) -- ([xshift=-5pt] c7.center) --([xshift=-5pt, yshift=5pt] c7.center);
\draw [-,blue] (b7.center) -- ([xshift=-5pt] b7.center) --([xshift=-5pt, yshift=5pt] b7.center);
\end{tikzpicture}
\begin{tikzpicture}
\matrix [column sep=0cm, row sep = 0.1cm] 
{
\node(a1) {$\bullet$}; & & & &  & &  & \node(a2) {\phantom{$\circ_1$}}; & & \\
\node(b1) {$\bullet$}; & & & &  & & \node(b2) {\phantom{$\circ_2$}};  & \node(a3){}; & & &\node(a4) {$\times$};  \\
\node(c1) {$\bullet$}; & & & &\node(c2) {\phantom{$\circ_1$}}; &&  \node(b3) {};&  \node(b4) {\phantom{$\circ_2$}};  & \\
\node(d1) {$\bullet$}; & & &\node(d2){\phantom{$\circ_3$}}; &\node(c3){}; & \node(c4) {\phantom{$\circ$}}; & & \node(b5) {};& \node(b6) {\phantom{$\circ$}};&  \\
& & &\node(d3){}; & \node(d4){\phantom{$\circ_3$}}; & \node(c5) {}; & &  \node(c6) {\phantom{$\circ_1$}}; & & \node[blue] {\hskip-30pt $\ddots$}; \\
\node {$\vdots$};& & & & &\node[green] {\hskip-20pt$\ddots$};   & &  &\node[red]{$\hskip-25pt\ddots$};  & \node(b7) {$\times$};  \\
\node(e1) {$\bullet$}; & \node(e2){$\times$};& & \node{$\dots$};& &   \node(d7) {$\times$};   & & & \node(c7) {$\times$}; & & \\
&\node {$b_1$};& & & & \node{\hskip5pt $b_{n-2}$};    & & &\node {\hskip5pt$b_{n-1}$}; & \node {\hskip5pt$b_n$}; \\
}; 
\draw [-,red] (a1.center) -- (a2.center) -- (a3.center) -- (a4.center);
\draw [-,blue] (b1.center) -- (b2.center) -- (b3.center) -- (b4.center) -- (b5.center) -- (b6.center)--([yshift=-15pt] b6.center);
\draw [-,red] (c1.center) -- (c2.center) -- (c3.center) -- (c4.center) -- (c5.center) -- (c6.center)--([yshift=-15pt] c6.center);
\draw [-,green] (d1.center) -- (d2.center) -- (d3.center) --(d4.center)--([yshift=-15pt] d4.center);
\draw [-,black] (e1.center) -- (e2.center);
\draw [-,green] (d7.center) -- ([xshift=-5pt] d7.center) --([xshift=-5pt, yshift=5pt] d7.center);
\draw [-,red] (c7.center) -- ([xshift=-5pt] c7.center) --([xshift=-5pt, yshift=5pt] c7.center);
\draw [-,blue] (b7.center) -- ([xshift=-5pt] b7.center) --([xshift=-5pt, yshift=5pt] b7.center);
\node [ ] at  ([xshift=10pt, yshift=7pt] a1) {$L_0$};
\node [ ] at  ([xshift=10pt, yshift=7pt] b1) {$L_1$};
\node [ ] at  ([xshift=10pt, yshift=7pt] c1) {$L_2$};
\node [ ] at  ([xshift=10pt, yshift=7pt] d1) {$L_3$};
\node [ ] at  ([xshift=6pt, yshift=7pt] e1) {$L_n$};
\end{tikzpicture}
\caption{On top is a  schematic image of an MLQ projecting to $s_1w_0$ of type II. Below is the corresponding set $\mathcal L_{I\! I}$ of non-intersecting lattice paths as in the proof of Theorem \ref{T:k=1,2}. }
\label{F:k1}
\end{figure}
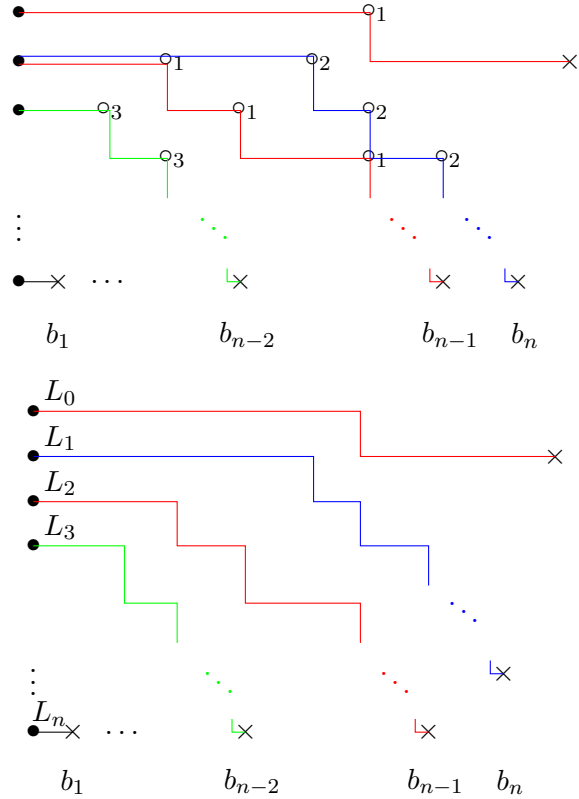

By the Lindst\"om-Gessel-Viennot lemma all such sets of non-intersecting paths are counted by the determinant of the $(n+1)\times (n+1)$-matrix:

\[
D_{I\! I}= \left[\begin{array}{ccccc}
1&  \cdots &\binom{b_1+j-1}{j-1}&\cdots &\binom{b_1+n}{n}\\
\vdots& &\vdots & \vdots &\vdots\\
1& \cdots &\binom{b_{n-1}+j-1}{j-1}&\cdots &\binom{b_{n-1}+n}{n}\\
\hline
0& \cdots &\binom{b_n+j-2}{j-2}&\cdots &\binom{b_n+n-1}{n-1}\\
0 &\cdots & 0& 1 & \binom{N}{1}\\
\end{array}\right].
\]

Hence the total number of MLQ's is
\[G_{s_1w_0}(b_1,\dots,b_n;N)=\det D_{I\! I}+\det D_I - G_{w_0}.
\]
Expanding $D_{I\! I}$ along the bottom row gives $\ \det D_{I\! I}=N\cdot \det A_1-\det D_I$ and the statement for $k=1$ follows.

\bigskip

The case $k=2$ is very similar but more complicated. This time there are three different types of MLQ's.

\begin{itemize}
\item Type I: no bully path wraps.
\item Type II: $P_2$ wraps in row 3. 
\item Type III: $P_2$ wraps in row $3$ and $P_1$ wraps in row $2$.
\end{itemize}

These are all the cases since if $P_1$ wraps then $P_2$ must also wrap for the 1 to end up last in the permutation.
For each type we give an injection to a set of tuples of non-intersecting paths, which can be counted using the Lindstr\"om-Gessel-Viennot Lemma.

Type I: Assume $P_1$ and $P_2$ intersect for the first time at $(r,c)$. No other bully paths touch.
We map such MLQ's injectively to $\mathcal L_I=\{L_n,\dots, L_1\}$, with starting positions
$\mathcal S_I=\{(n,0),\dots,(3,0),(1,0),(0,0)\}$ and ending positions $\mathcal M_I=\{(n,b_1),\dots,(n,b_{n-2}),(n-1,b_{n-1}),(n-1,b_n)\}$. 
Let $L_1$ be $P_1$ translated one step upwards. Let $L_2$ be the translation one step upwards of the concatenation of $P_2$ 
from $(2,0)$ to $(r,c)$, and $P_3$ from $(r,c)$ to $(n,b_{n-1})$. Let $L_3$ be the concatenation of the remaining pieces of $P_2$ and $P_3$ and $L_i=P_i$, 
for $4\le i\le n$. The total number of such $n$-tuples of non-intersecting paths is the determinant of the following $n\times n$-matrix:

\[
E_I= \left[\begin{array}{cccc|cc}
1&\ddots&\vdots &\iddots &\vdots&\vdots\\
\vdots &\cdots &\binom{b_i+j-1}{j-1}&\cdots &\binom{b_i+n-1}{n-1}&\binom{b_i+n}{n}\\
1&\iddots & \vdots  &\ddots &\vdots&\vdots \\
\hline
0&\cdots &\binom{b_{n-1}+j-2}{j-2}&\cdots &\binom{b_{n-1}+n-2}{n-2}&\binom{b_{n-1}+n-1}{n-1}\\[0.3em]
0&\cdots &\binom{b_n+j-2}{j-2}&\cdots &\binom{b_n+n-2}{n-2}&\binom{b_n+n-1}{n-1}\\
\end{array}\right].
\]
As in the type I case above we must subtract those where the vertical distance between $L_2$ and $L_3$ is at least 2 all the time, which again is $G_{w_0}$.

Type II: Let $(r,c)$ be the point where $P_1$ and $P_2$ intersect the first time. In the beginning of row 3, $P_2$, which wraps, and $P_3$ will overlap.
Here $L_0$ is the translation two steps upwards of the concatenation of $P_1$ to $(r,c)$ and $P_2$ 
from $(r,c)$ to $(3,N-1)$ before it wraps. $L_1$ is the translation one step upwards of the concatenation of $P_2$ from $(2,0)$ to $(r,c)$ and $P_1$ from $(r,c)$ to $(n,b_n)$. Let $L_2$ be the translation one step upwards of $P_3$ and let $L_3$ be $P_2$ after it has wrapped, that is, from $(3,0)$ to $(n,b_{n-2})$.

We map these MLQ's injectively to $\mathcal L_{I\! I}=\{L_n,\dots, L_1,L_0\}$, with starting positions
$\mathcal S_{I\! I}=\{(n,0),\dots,(2,0),(1,0),(-1,0)\}$ and ending positions $\mathcal M_{I\! I}=\{(n,b_{1}),\dots, (n,b_{n-2}), (n-1,b_{n-1}),(n-1,b_n),(1,N-1)\}$. 
To count the number of MLQ's of type II we have to subtract of the sets $\mathcal L_{I\! I}$ where the vertical distance between $L_0$ and $L_1$ is 2 or more in each column, which can be counted by lowering the start and endpoints of $L_0$ by one. This means that the number of MLQ's of type II is counted by the difference $\det(E_{I\! I})-\det(E_{I\! I}')$, where $E_{I\! I}$, $E_{I\! I}'$ are the following $(n+1)\times (n+1)$-matrices:

\[
E_{I\! I}= \left[\begin{array}{cccc|c}
1&\ddots&\vdots &\iddots &\vdots\\
 &\cdots &\binom{b_i+j-1}{j-1}&\cdots &\binom{b_i+n+1}{n+1}\\
1&\iddots & \vdots  &\ddots &\vdots \\
\hline
0& \cdots &\binom{b_{n-1}+j-2}{j-2}&\cdots &\binom{b_{n-1}+n}{n}\\
0&\cdots &\binom{b_n+j-2}{j-2}&\cdots  &\binom{b_n+n}{n}\\[0.3em]
0& \cdots & \!\! \cdots \ \ 0 & 1 &\binom{N+1}{2}
\end{array}\right],\]
\[
E_{I\! I}'= \left[\begin{array}{ccccc|c}
1&\ddots&\vdots &\iddots &\vdots&\vdots\\
&\cdots &\binom{b_i+j-1}{j-1}&\cdots &\binom{b_i+n-1}{n-1}&\binom{b_i+n}{n}\\
1&\iddots & \vdots  &\ddots &\vdots&\vdots \\
\hline
0&\cdots &\binom{b_{n-1}+j-2}{j-2}&\cdots &\binom{b_{n-1}+n-2}{n-2}&\binom{b_{n-1}+n-1}{n-1}\\
0 &\cdots &\binom{b_n+j-2}{j-2}&\cdots  &\binom{b_n+n-2}{n-2}&\binom{b_n+n-1}{n-1}\\[0.3em]
0 &\cdots& \!\! \cdots \ \ 0 & 1 & \binom{N}{1}&\binom{N+1}{2}
\end{array}\right].
\]

Type III:
The MLQ's where both $P_1$ and $P_2$ wrap around can similarly be bijectively mapped to $n+2$-tuples of non-intersecting 
paths starting in positions
$\mathcal S_{I\! I\! I}=\{(n,0),\dots,(1,0),(0,0), (-1,0)\}$ and ending positions $\mathcal M_{I\! I\! I}=\{(n,b_{1}),\dots,(n,b_{n-2}), (n-1,b_{n-1}), (n-1,b_{n})$, $(1,N-1),(0,N-1)\}$. 
These are counted by the determinant of the matrix:

\[
E_{I\! I\! I}= \left[\begin{array}{cccccc}
1&\ddots&\vdots &&\iddots & \\
1&\cdots &\binom{b_i+j-1}{j-1}&&\cdots &\\
1& \iddots  &\vdots & &\ddots &\\
\hline
0&\cdots &\binom{b_{n-1}+j-2}{j-2}&&\cdots & \\
0 &\cdots &\binom{b_n+j-2}{j-2}&&\cdots  & \\
0 & \cdots &0 & 1 & \binom{N}{1} & \binom{N+1}{2} \\[0.3em]
0 & \cdots &0 & 0 & 1 & \binom{N}{1} \\
\end{array}\right].
\]
Expanding the matrices $E_I,E_{I\! I},E_{I\! I}', E_{I\! I\! I}$ along the bottom rows most terms of the determinants cancel to give the claimed result.
\end{proof}

One way to prove Conjecture \ref{Conj:Pdfoneaway} would be to first establish the following formula for the discrete chain.

\begin{conjecture}\label{Conj:Oneaway} For  $N\ge n>k\ge 1$ and $0\le b_1<\dots<b_n\le N-1$ the number of MLQ's with bottom row $s_kw_0$ is
\[G_{s_kw_0}(b_1,\dots,b_n;N)=\binom{N}{k}\det A_k - G_{w_0}.
\]
\end{conjecture}

\medskip
The following more general conjecture for commuting simple reflections implies Conjecture \ref{Conj:Pdfmanyaway} .
For  a partition $\mbf{k}=(k_1\ge\cdots k_r\ge 1)$, let $\mbf{k'}$ denote the conjugate partition. For a subset $S\subseteq [r]$, let $\mbf{k}(S)$ be the partition consisting of the parts $k_{i}, i\in S$. With $\mbf{k}(S)'$ we denote the conjugate of $\mbf k(S)$.
Let $A_S$ be the matrix with entries $\binom{b_i+j-1-k_{n+1-i}(S)'}{j-1-k_{n+1-i}(S)'}$, where $k_{i}(S)'=0$ if $i>k_{\min S}$.

\begin{conjecture}\label{Conj:Manyaway} For  $N\ge n$ and $\mbf{k}$ such that $n>k_1>k_2+1>k_3+2>\dots >k_r+r-1>r-1$ and $0\le b_1<\dots<b_n\le N-1$, we have
\[
	G_{s_{k_1}\dots s_{k_r}w_0}(b_1,\dots,b_n;N)= \sum_{S\subseteq [r]}  (-1)^{r-|S|} \prod_{i\in S} \binom{N}{k_i}\det A_S.
\]
\end{conjecture}

Note that $A_\emptyset$ is the Vandermonde matrix so it specializes to Conjecture \ref{Conj:Oneaway} for $r=1$. We have checked this conjecture for $n \leq 5$ (this only includes cases with $r \leq 2$).

\subsection{Probability of given permutation}
One obvious question to ask about the distribution $\SD_n$ is the probability $p_{\pi}$ that the particles form a certain permutation $\pi$.
We can compute the exact probability of the reverse permutation $w_0$.

\begin{theorem}\label{T:ExactProb}
The probability that the particles form the reverse permutation $w_0$ is
\[
	p_{w_0}=\frac{1}{\prod_{k=1}^{n-1} {2k+1 \choose k+1}}.
\] 
\end{theorem}
\begin{proof}
As in the proof of \refP{P:w0}, each row $i$ of an MLQ corresponding to $w_0$ is labeled by the reverse permutation of length $i$. Furthermore, the positions of the boxes on row $i-1$ interleave the positions of the boxes on row $i$. Such placements (in the sense of Remark \ref{R:count}) have been studied before in other contexts and are called Gelfand-Tsetlin patterns, see \cite{oeis}.
The enumeration of all such patterns seems to have been done first in \cite{T}, where it is proven that the number of such patterns is
\[ \frac{\binom{n+1}{2}!\prod_{i=1}^{n-1} i!}{\prod_{i=1}^{n-1} (2i+1)!}.
\]

Think of the $\binom{n+1}{2}$ boxes in the MLQ as chosen in the interval $[0,1)$, and then selecting which boxes end up on which line. By Remark \ref{R:count} we conclude that the $p_{w_0}$ is the number of placements yielding $w_0$ divided by the total number, $\binom{\binom{n+1}{2}}{1,\dots,n}$, of placements. This gives the stated formula.
\end{proof}

We have computed $p_{\pi}$ for all permutations of length $n\le 6$. For $n=2,3$ they are as follows
\renewcommand{\arraystretch}{1.2}
\begin{tabular}{l || c | c | c | c | c | c | c | c }
$\pi$ & 12 & 21 & 123 & 231 & 312 & 132 & 213  & 321 \\
\hline 
$p_{\pi}$ & $\frac{2}{3}$  & $\frac{1}{3}$ & $\frac{25}{60}$ & $\frac{13}{60}$ & $\frac{10}{60}$ & $\frac{5}{60}$ & $\frac{5}{60}$  & $\frac{2}{60}$   \\
\end{tabular}.

\smallskip
Unfortunately we cannot see any obvious general pattern, regardless of whether or not  we mod out by the cyclic action. For example, in general, $p_\pi/p_{w_0}$ is not an integer and the chain is not symmetric: $p_{w_0\pi w_0} \neq p_\pi$ in general. We note, however, that $p_{w_0}$ appears to be the smallest of the probabilities, and $p_{id}$ the largest.


\section{Correlations}\label{S:Correlations}

Even though the stationary probability $p_\pi$ of a given permutation $\pi$ seems difficult to describe in general, the correlation of two adjacent elements seems to exhibit interesting patterns for the continuous TASEP. Let $c_{i,j}(n)=\P(w_a=i,w_{a+1}=j$, for some $a)$, where $a+1$ is modulo $n$. See Table \ref{Table:corr6} for the values of $c_{i,j}(6)$.

\begin{table}
\begin{tabular}{|r||r|r|r|r|r|r|}
\hline
$i\backslash j$& $1$& $2$& $3$& $4$& $5$& $6$\\
\hline
\hline
$1$& $0    $& $      1/2    $& $      1/6    $& $      2/15     $& $     6/55     $& $     1/11 $\\
\hline
$2$ &  $ 1/14      $& $    0     $& $   25/42   $& $       2/15   $& $       6/55     $& $     1/11 $\\
\hline
$3$ &  $5/42       $& $   1/21    $& $      0    $& $     19/30       $& $   6/55     $& $     1/11 $\\
\hline
$4$ &  $16/105     $& $    17/210     $& $     1/30      $& $    0     $& $   106/165    $& $      1/11 $\\
\hline
$5$ &  $ 68/385    $& $     81/770      $& $   19/330      $& $    4/165  $& $        0      $& $    7/11 $\\
\hline
$6$ &  $   37/77      $& $   41/154    $& $     34/231    $& $      5/66       $& $   1/33     $& $     0 $\\
\hline
\end{tabular}
\vskip2mm
\caption{Table showing $ c_{i,j}(n)$, for $n=6$.}\label{Table:corr6}
\end{table}

The most obvious observation is that the columns in the upper right part seem to be constant.
To be more precise:

\begin{conjecture} \label{Conj:column} For every $i+1<j$ we have $c_{i,j}(n)=n/{\binom{n+j}{2}}$.
\end{conjecture}

From this, it would also follow that $c_{n-1,n}=(n+1)/(2n-1)$ and $c_{1,2}=4/(n+2)$, since $\sum_i c_{i,j} = \sum_j c_{i,j} = 1$.

It seems the denominator is always a product of small primes. 
The data for $n\le 6$ suggest a conjecture covering all the $c_{i,j}$s.
Our main conjecture for the correlations in the continuous TASEP on a ring is the following.

\begin{conjecture} \label{Conj:Correlation}
For $n\ge 2$, we have the following two-point correlations at stationarity
\[
c_{i,j}(n)=\begin{cases}
\frac{n}{\binom{n+j}{2}}, &\text{if }  i+1<j\le n, \\
\frac{n}{\binom{n+j}{2}}+\frac{ni}{\binom{n+i}{2}}, &\text{if }  i+1=j\le n, \\
\frac{n}{\binom{n+j}{2}}-\frac{n}{\binom{n+i}{2}}, &\text{if }  j<i<n, \\
\frac{n(j+1)}{\binom{n+j}{2}}-\frac{n(j-1)}{\binom{n+j-1}{2}}-\frac{n}{\binom{2n}{2}}, &\text{if }  j<i=n.\\
\end{cases}
\]
\end{conjecture}

So, according to the conjecture, for any $j>1$ the most likely position is (directly) to the right of $j-1$. If $j$ is close to $n$, then this will happen roughly half the time.
For small $j$ ($1<j<n/3$) the second most likely position for $j$ is to the right of $n$.

\begin{remark}
According to the conjecture, $n$ is always a factor for the probabilities.
It is tempting to divide with $n$ and say that we are interested in the case when $w_1=i, w_2=j$.
This would, however, not be an equivalent formulation.
The spacing between the particles is not uniform, and hence the distribution of which particle is first in a given $[0,1)$ interval is not uniform.
\end{remark}

We can prove a few cases of this conjecture.
\begin{proposition}
The following two-point correlations hold for any $n\ge 3$.
\begin{enumerate}
\item{} $c_{2,1}=\frac{n}{\binom{n+1}{2}}-\frac{n}{\binom{n+2}{2}}=\frac{4}{(n+1)(n+2)}$
\item{} $c_{1,2}=\frac{n}{\binom{n+2}{2}}+\frac{n}{\binom{n+1}{2}}=\frac{4}{(n+2)}$
\item{} $c_{n,n-1}=\frac{n^2}{\binom{2n-1}{2}}-\frac{n(n-2)}{\binom{2n-2}{2}}-\frac{n}{\binom{2n}{2}}=\frac{3}{(2n-1)(2n-3)}$
\end{enumerate}
\end{proposition}

\begin{proof}
For the first two statements, we will use the process of the last row on $n$ boxes labeled $1,\dots,n$. Assume that the particles of classes 2 and 1 are positioned at $q_2$ and $q_1$ respectively in the preceding row.
By rotation we may assume that $q_2<q_1$.
The only way to obtain a 2 followed by a 1 after the process of the last row is to have exactly one particle in the interval $[q_2,q_1]$.
This is due to the fact that the class 1 particle will land at the first available position and after that the class 2 particle will do the same.
Let $y=1-(q_1-q_2)$.
We know by \refT{T:w0} that $g_{21}(q_1,q_2)=2(q_2-q_1) = 2(1-y)$.
If we think of this as the limit of the stationary distribution of TASEP, it is clear that the particles of classes higher than 2 will not influence the relative positions of 2 and 1. 
The probability that exactly one of the $n$ particles lies in the interval $[q_2,q_1]$ is $\binom{n}{1}(1-y)y^{n-1}$.
We thus obtain
\[c_{2,1}=\P(\text{2 followed by 1})=\int_0^1 2n(1-y)^2y^{n-1} dy=\frac{4}{(n+1)(n+2)}.
\]

The computation of $c_{1,2}$ is similar.
This time, there are three possibilities of getting a 1 followed by a 2: either there are no particles in the interval $[q_2,q_1]$, or all the particles are in this interval, or all particles but one are in the interval.
Summing these three integrals gives the desired formula. Note that the method used above could in principle be extended to $c_{i,j}$ for $i,j\le x$, if we know $g_\pi$ for $\pi\in S_x$, but it quickly becomes intractable. 

The computation of $c_{n,n-1}$ is more involved. We use continuous multiline queues. As discussed in Remark \ref{R:count}, to determine $c_{n,n-1}$, it suffices to count the number of placements $(a_{ij})$ of the boxes in the multiline queue that give a bottom row starting with $n, n-1$ (by rotation we may assume that the leftmost box is labelled $n$). We will refer to such placements as {\it valid}. So to compute $c_{n,n-1}$, we will compute the ratio of valid placements to all placements.

Note that the values of $a_{ij}$ for $i \leq n-3$ do not influence whether an assignment is valid or not. Therefore we can instead count the number of placements $(b_{ij})$ for $n-2 \leq i \leq n, j \leq i$, where $b_{ij}$ is the number of boxes to the right of box $(i,j)$ \emph{in the bottom three rows}, including the box itself. Recall that $j$ is the number of the box counting from the right. Such a placement $(b_{ij})$ is valid if it comes from a valid placement $(a_{ij})$. 

It is easy to check that being valid amounts to the following systems of inequalities. 

\newcommand{\rle}[1]{\text{\begin{turn}{#1}$<$\end{turn}}} 


\[
\begin{array}{ccc}
  b_{n,1}   & < \dots  < & b_{n,n-2}     \\
  \rle{270} &            & \rle{270}    \\
  b_{n-1,1} & < \dots  < & b_{n-1,n-2}  \\
  \rle{270} &            & \rle{270}    \\
  b_{n-2,1} & < \ldots < & b_{n-2,n-2}  \\
\end{array} \text{\  and\  }
\begin{array}{ccccc}
    &&&&  \\
    &&&& \\
 b_{n-1,n-2} & < & b_{n,n-1}&& \\
 \rle{270}   &&\rle{270}&& \\
 b_{n-2,n-2} & < & b_{n-1,n-1}  &< &b_{n,n} \\
\end{array}
\]

The left system of inequalities comes from the fact that there is no wrapping in the bottom three rows and the second set of inequalities comes from the word starting with $n,n-1\dots$.
From the inequalities it follows directly that $b_{n,n}=3n-3$ and $b_{n-1,n-1}=3n-4$. If $b_{n,n-1}=2n+i-1$ then $b_{n-2,j}=2n+j-3$ for all $i< j\le n-2$. The remaining entries, which form the set $\{1,\dots,2n-3+i\}$, form a standard Young tableau (transpose the figure of the inequalities above) with columns of length $n-2,n-2,i$.
Let $SYT_{n-2,n-2,i}$ denote the number of standard Young tableaux of this type. We refer to \cite{EC2} for basic facts about standard Young tableaux including the hook-length formula from which one may deduce that $SYT_{n-2,n-2,i}=\frac{(2n-4+i)!(n-i)(n-i-1)}{i!n!(n-1)!}$. The rotation gives a factor of $3n-3$ and we get

\begin{equation*}
c_{n,n-1}=\frac{(3n-3)\sum_{i=0}^{n-2} SYT_{n-2,n-2,i}}{\binom{3n-3}{n,n-1,n-2}}=\frac{(n-2)!\sum_{i=0}^{n-2} \frac{(2n-4+i)!(n-i)(n-i-1)}{i!}} {(3n-4)!}.
\end{equation*}

Now, expand $(n-i)(n-i-1)=n(n-1)-(2n-2)i+i(i-1)$ and use $\sum_{i=0}^{x} \binom{y+j}{j}=\binom{y+x+1}{x}$ for $j=i,i-1$ and $i-2$. Collect all terms and we get 
$\frac{3}{(2n-1)(2n-3)}$, as desired. 
\end{proof}

\section{Correlation function for initial decreasing sequence}\label{S:initial}

In this section, we prove a formula for the probability that a word sampled from the discrete TASEP starts with a given decreasing word. By remarkable coincidence it is the same formula as in \refP{P:w0}. Fix the length $N$ of the ring and let $\mbf{m} = (\underbrace{1,\dots,1}_{N})$.
Suppose $u$ is picked from the stationary distribution of the $\mbf{m}$-TASEP. 
We now ask, what is the probability that $u$ has some fixed word as a prefix? In general the answer appears to be complicated (see \cite{AL2} for prefixes of length at most $3$).
However in the case of a word of the type $x_n x_{n-1} \dots x_2$ where $x_n > x_{n-1} > \dots > x_2$, we show that there is a simple answer to this question. This theorem answers \cite[Conjecture 8.1]{AL2}.

\begin{theorem} \label{T:initial} Suppose $u$ is picked from the stationary distribution of the $\mbf{m}$-TASEP. 
Fix $N \geq x_n > x_{n-1} > \dots > x_2 \geq 1$. Then, the probability $f_\pi(x_n, \dots, x_2)$ ($\pi$ for "permutation") that for some word $v$, $u = x_n x_{n-1} \dots x_2 v$, is 
\[
	\frac{1}{\prod_{i=1}^{n-1} {N\choose i}}
	\det\left[{x_{i+1} \choose j-1}\right]_{i,j=1}^{n-1}.
\] 
\end{theorem} 

\begin{remark}\label{R:5} One way to think about this theorem is to consider a state of the TASEP as a permutation matrix of size $N$ with 1's in positions $(i,\pi(i))$. \refT{T:initial}  gives the probability that the first $n-1$ rows have the 1's in positions $(1,x_n),\dots, (n-1, x_2)$. (By cyclic invariance the same is true for any $n-1$ consecutive rows.) On the other hand, \refP{P:w0} states that the probability of the first \emph{columns} of the state matrix having 1's in positions $(x_n,n-1), \dots ,(x_2, 1)$ is exactly the same as the expression in \refT{T:initial}. This is because the probability of the position of the smallest labels does not change if we change all labels $n,\dots,N$ to just $n$. This interesting equality does not carry over to other patterns in general.
\end{remark}

\begin{proof}

For a vector $(m_1, \dots, m_n)$ of non-negative integers with sum $N$, write $N_i = (M_{i-1}, M_i]$ where $M_0= 0$ and 
$M_i = \sum_{j \leq i} m_j$, $1\le i \le n$ so that $[N] = \{1,2,\dots, N\}$ is the disjoint union of the $N_i$'s. We also assume $m_i\ge 1, i\ge 1$, so only the first interval may be empty.

Now, write $f_w(m_n, m_{n-1}, \dots, m_1)$ ($w$ for "word") for the probability that a TASEP distributed word of type $\mbf{m}$ starts with the word $n(n-1)\dots 32$.

Clearly,
\begin{equation}
\label{fw_in_terms_of_fpi}
f_w(m_n, \dots, m_1) = \sum_{x_n \in N_n} \dots \sum_{x_2\in N_2} f_\pi(x_n, \dots, x_2).
\end{equation}

For given $n,N$, the set of possible vectors $(m_n, \dots, m_1)$ and the set of possible vectors $(x_n, \dots, x_2)$ both have size $\binom{N}{n-1}$ and a bijection is given by setting $x_i=M_{i-1} +1$. The relation \eqref{fw_in_terms_of_fpi}
thus amounts to multiplication with a quadratic matrix. Ordering $(m_n, \dots, m_1)$ by lexicographic order and 
$(x_n, \dots, x_2)$ by backward revlex order the matrix will be lower triangular with 1's on the diagonal and thus invertible. So, for a fixed $x_n, \dots, x_2$, the value of $f_\pi(x_n, \dots, x_2)$ can be computed from the values of all $f_w(m_n, \dots,m_1)$.
Thus, to prove the theorem, it is sufficient to show that equation \eqref{fw_in_terms_of_fpi} is satisfied (for all $\mbf{m}$) when substituting the claimed formula for $f_\pi$, that is,

\begin{equation}
\label{fw_toprove}
f_w(m_n, \dots, m_1) =
\frac{1}{\prod_{i=1}^{n-1} {N\choose i}}
\sum_{x_n\in N_n} \dots \sum_{x_2 \in N_2} \det\left[{x_{i+1} \choose j-1}\right]_{i,j=1}^{n-1}.
\end{equation}

We will now make a series of manipulations to the right hand side of \eqref{fw_toprove}.
First, note that $x_{i+1}$ only occurs in row $i$ and use the multilinearity of the determinant to move each sum inside its respective row. We get

\[
\frac{1}{\prod_{i=1}^{n-1} {N\choose i}}
\det\left[\sum_{x_{i+1}\in N_{i+1}} {x_{i+1} \choose j-1}\right]_{i,j=1}^{n-1} =
\]\[
\frac{1}{\prod_{i=1}^{n-1} {N\choose i}}
\det\left[{M_{i+1}+1 \choose j} - {M_{i}+1 \choose j}\right]_{i,j=1}^{n-1}.
\]

%
Letting each row be replaced with the sum of the rows (weakly) above we obtain:
\[
	\det\left[{M_{i+1}+1\choose j}-{M_{i}+1\choose j}\right]_{i,j=1}^{n-1} =
	\det\left[{M_{i+1}+1\choose j}-{M_1+1\choose j}\right]_{i,j=1}^{n-1}.\]

If we in the $n\times n$ matrix $\left[{M_i+1\choose j-1}\right]_{i,j=1}^{n}$ subtract the top row from every other row, 
and expand along the first column we get the identity 

\[
	\det\left[{M_i+1\choose j-1}\right]_{i,j=1}^{n}=\det\left[{M_{i+1}+1\choose j}-{M_1+1\choose j}\right]_{i,j=1}^{n-1}.\]

Let $F_w(m_n, m_{n-1}, \dots, m_1) = \prod_{i=1} ^n {N\choose M_i} \cdot f_w(m_n,m_{n-1}, \dots, m_1)$ be the number of multi-line queues whose bottom row starts with $n(n-1)\dots 32$ and has type $\mbf{m}$. 

To prove equation \eqref{fw_toprove}, we need to show that  

\[
	F_w(m_n, m_{n-1}, \dots, m_1) =
	\prod_{i=1}^n \frac{{N\choose M_i}}{{N\choose i-1}}
	\det\left[{M_i+1\choose j-1}\right]_{i,j=1}^n.
\]

Move the product in	the numerator into the rows of the matrix and the product in the denominator into the columns. Simplify the resulting expression $\frac{{N\choose M_i}{M_i+1 \choose j-1}}{{N\choose j-1}} = \frac{M_i+1}{N+2-j}{N+2-j\choose	M_i+2-j}$ and move the factors depending only on $i$ respectively $j$ out through the rows and columns again. Recall that $M_n=N$.  We get 

\[
\prod_{i=1}^n \frac{M_i +1}{N+2-i}
\det\left[ {N+2-j\choose M_i +2-j} \right]_ {i,j=1} ^n = 
\]\[
\prod_{i=1}^{n-1} \frac{M_i +1}{N+1-i}
\det\left[ {N+2-j\choose M_i +2-j} \right]_ {i,j=1} ^n.
\]

This matrix has bottom row with all ones. We replace column $j$ in this matrix with the difference of column $j$ and 
column $j+1$. Then we use ${N+2-j\choose M_i +2-j} - {N+1-j\choose M_i +1-j} = {N+1-j\choose M_i +2-j}$ and expand 
the determinant along the bottom row which yields

\begin{equation}\label{E:2}
\prod_{i=1}^{n-1} \frac{M_i +1}{N+1-i}
\det\left[ {N+1-j\choose M_i +2-j} \right]_ {i,j=1} ^{n-1}.
\end{equation}

It remains to show that the last expression equals $F_w(m_n, \dots, m_1)$. Consider a multiline queue counted by $F_w(m_n, \dots,m_1)$.  It has $n-1$ rows, indexed $1, \dots, n-1$ (we use matrix notation: $(i,j)$ refers to row $i$ and column $j$). In rows $1, 2, \dots, (n-1)$, there are $m_1, m_1+m_2, \dots, N - m_n$ particles.
It is easy to see that the sites in the "triangle" $(n-1, 1), \dots, (n-1, n-2)$; $(n-2, 2), \dots, (n-2, n-2)$; $\dots$; $(2,n-2)$ are filled with boxes, and that a box in such a site $(i,j)$ is labeled $n-j$. Recall that we number the columns from $0$ to $N-1$. We now consider the boxes in the multiline queue that are not part of this triangle. Since the word starts with the descending sequence, no bully paths can wrap. 

Denote by $z_{i,j}$ the distance from the right end of the multiline queue of the $j$th box from the right in the $(n-i)$th row. 
That is, if the $j$th particle from the right in the $(n-i)$th row is at $(i, r)$, we let $z_{i,j} = N-r$.
We only include $(i,j)$ referring to boxes not in the triangle mentioned above.
The numbers $z_{i,j}$ must form a semi-standard Young tableau (SSYT) of shape $\lambda$ where the conjugate partition is $\lambda'_i = M_{n-i} - (n-i-1)$ for $1 \leq i \leq n-1$. Moreover, this is a bijection from the MLQs counted by $F_w(m_n, \dots, m_1)$ to SSYT of shape $\lambda$ with entries in $[t]$, where $t = N-n+1$.

\begin{example}
Here, $N = 13$, $n = 5$, $\mbf{m} = (2, 2, 2, 3)$, $(M_1, \dots, M_5) = (2, 4, 6, 9, 13), t = 9, \lambda' = (6,4,3,2)$.
The multiline queue

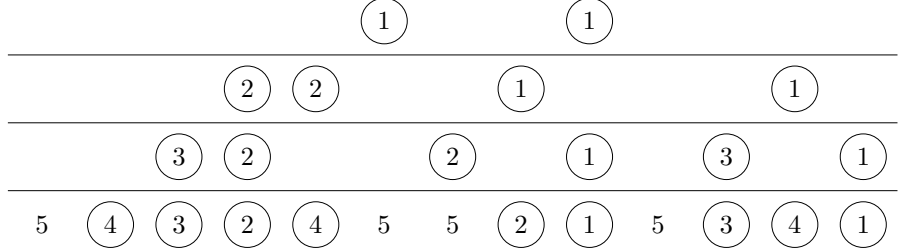
\scalebox{0.9}{
\begin{tikzpicture} 
\node at (1,1){$5$};
\node[circle,draw=black] at (2,1){$4$}; 
\node[circle,draw=black] at (3,1){$3$}; 
\node[circle,draw=black] at (4,1){$2$};
\node[circle,draw=black] at (5,1){$4$};
\node at (6,1){$5$};
\node at (7,1){$5$};
\node[circle,draw=black] at (8,1){$2$};
\node[circle,draw=black] at (9,1){$1$};
\node at (10,1){$5$};
\node[circle,draw=black] at (11,1){$3$};
\node[circle,draw=black] at (12,1){$4$};
\node[circle,draw=black] at (13,1){$1$};
\draw(0.5,1.5)--(13.5,1.5);
\node[circle,draw=black] at (3,2) {$3$};
\node[circle,draw=black] at (4,2) {$2$};
\node[circle,draw=black] at (7,2) {$2$};
\node[circle,draw=black] at (9,2) {$1$};
\node[circle,draw=black] at (11,2){$3$};
\node[circle,draw=black] at (13,2){$1$};
\draw(0.5,2.5)--(13.5,2.5);
\node[circle,draw=black] at (4,3) {$2$};
\node[circle,draw=black] at (5,3) {$2$};
\node[circle,draw=black] at (8,3) {$1$};
\node[circle,draw=black] at (12,3){$1$};
\draw(0.5,3.5)--(13.5,3.5);
\node[circle,draw=black] at (6,4) {$1$};
\node[circle,draw=black] at (9,4){$1$};
\end{tikzpicture}
}

corresponds to the tableau  

${\scriptsize \Yvcentermath1\young (1125,2368,359,57,6,9)}$. 

 \smallskip
 So, for example, the first row of the SSYT describes the positions from the right of the rightmost box in the four lines in the queue.
\end{example}

Now we are in a position to finish our argument. By the definition of Schur function, the number of SSYT of shape $\lambda$ with entries in $[t]$ is $s_\lambda(1^t)$. Recall the hook-content formula and the Jacobi-Trudi identity, see  e.g. \cite{EC2}.

\begin{lemma}
The number of SSYT of shape $\lambda$ and entries in $[t]$ equals
\[
\prod_{r\in\lambda} \frac{t+c_\lambda(r)}{h_\lambda(r)} = s_\lambda(1^t) = \det\left[{t \choose \lambda'_i-i+j}\right] = \det\left[{t+j-1\choose \lambda'_i - i +j}\right],
\]
where for a box $r = (i,j)$ in the Ferrers diagram of $\lambda$, we let $c_\lambda(r) = j-i$ and $h_\lambda(r) = \lambda_i + \lambda'_j - i - j + 1$.
\end{lemma}

The first two equalities are well-known, and the last is easily obtained by column operations. 

So, by the lemma, 
\[ F_w(m_n,\dots,m_1)= 	\prod_{r\in\lambda} \frac{N-n+1 + c_\lambda(r)}{h_\lambda(r)}.
\]

Now let $s = t+1=N+1-n$ and $\mu'_i = \lambda'_i + 1$ for $1 \leq i \leq n-1$. We think of $\mu$ as $\lambda$ with a row added on top. It is easy to see that our determinant in \eqref{E:2} equals (after reversing the numbering of rows and columns)
\[\prod_{i=1}^{n-1} \frac{M_i +1}{N+1-i} 
	\det\left[ {s+j-1\choose \mu'_i -i +j} \right]_ {i,j=1} ^{n-1},
\]

which by the lemma (temporarily letting $\mu$ and $s$ play the roles of $\lambda$ and $t$) equals 

\[
	\prod_{i=1}^{n-1} \frac{M_i +1}{N+1-i} 
	\prod_{r\in\mu} \frac{N-n+2 + c_\mu(r)}{h_\mu(r)}
\]

To prove that $F_w(m_n, \dots, m_1)$ equals the expression \eqref{E:2}, it thus remains to show that
\[\prod_{i=1}^{n-1} \frac{M_i +1}{N+1-i} 
	\prod_{r\in\mu} \frac{N-n+2 + c_\mu(r)}{h_\mu(r)} =
	\prod_{r\in\lambda} \frac{N-n+1 + c_\lambda(r)}{h_\lambda(r)},
\]
or, in terms of $\lambda'$, 
\[\prod_{r\in\lambda} \frac{N-n+1 + c_\lambda(c)}{h_\lambda(c)} =
	\prod_{i=1}^{n-1} \frac{\lambda'_i+n-i}{N+1-i} 
	\prod_{r\in\mu} \frac{N-n+2 + c_\mu(c)}{h_\mu(c)}.
\]

This is easily checked.
\end{proof}

\end{document}